\documentclass{amsart}

\usepackage{hyperref}
\usepackage[usenames,dvipsnames]{xcolor}

\usepackage{amsmath, amsthm, amssymb, mathrsfs, calc}
\usepackage{array}
\usepackage{mathtools}
\usepackage{enumitem}
\usepackage{url}
\usepackage{setspace}
\usepackage{mleftright}
\usepackage{a4}
\usepackage{framed}
\usepackage{marginnote}
\usepackage[right]{lineno}

\makeatletter
\newtheorem*{rep@theorem}{\rep@title}
\newcommand{\newreptheorem}[2]{%
\newenvironment{rep#1}[1]{%
 \def\rep@title{#2 \ref{##1}}%
 \begin{rep@theorem}}%
 {\end{rep@theorem}}}
\makeatother

\newtheorem{prop}[equation]{Proposition}
\newtheorem{theorem}[equation]{Theorem}
\newtheorem{cor}[equation]{Corollary}
\newtheorem{lem}[equation]{Lemma}

\theoremstyle{definition}

\newtheorem{remarks}[equation]{Remarks}

\numberwithin{equation}{section}

\newcommand{\R} {\mathbb{R}} \newcommand{\Z} {\mathbb{Z}} \newcommand{\C} {\mathbb{C}}

\providecommand*{\eu}%
{\ensuremath{\mathrm{e}}}
\providecommand*{\iu}%
{\ensuremath{\mathrm{i}}}

\title[Statistics in conjugacy classes]{Statistics in conjugacy classes in free groups}

\author{George Kenison}
\address{George Kenison, School of Mathematics, University of Bristol, University Walk, Bristol BS8 1TW, U.K.}
\email{george.kenison@bristol.ac.uk}

\author{Richard Sharp}
\address{Richard Sharp, Mathematics Institute, University of Warwick,
Coventry CV4 7AL, U.K.}
\email{r.j.sharp@warwick.ac.uk}
\begin{document}

\begin{abstract} In this paper, we establish statistical results for 
a convex co-compact action of a free group on a CAT($-1$) space
where we restrict to a non-trivial conjugacy class in the group.
In particular, we obtain a central limit theorem where the variance is twice the variance 
that appears when we do not make this restriction.
\end{abstract}

\maketitle

\section{Introduction and results}
Let $\Gamma$ be a free group on $p \geq 2$ generators acting convex co-compactly on a 
$\mathrm{CAT}(-1)$ space $(X,d)$ (i.e the quotient of the intersection of $X$ and the convex hull of the 
limit set of $\Gamma$ is compact). There has been considerable work in trying to understand the 
statistics of such an action.
For example, the following result (a particular case of the \v{S}varc--Milnor lemma) is well-known. 
Fix a free generating set $\mathcal A = \{a_1,\ldots,a_p\}$ and 
let $|\cdot|$ denote word length on $\Gamma$ with respect to $\mathcal A$.
Then, for an arbitrary base point $o\in X$, there exist constants $C_1,C_2>0$ such that
\[
C_1|x| \leq d(o,xo) \leq C_2|x|  \eqno(1.1) 
\]
for all $x \in \Gamma$.
Thus $|x|$ and $d(o,xo)$ are comparable quantities and it is natural to ask if more precise
estimates hold, at least typically or on average.

One such result is the following.
Write $\Gamma_n := \{x \in \Gamma \mbox{ : } |x|=n\}$. Then
the averages
\[
\frac{1}{\#\Gamma_n} \sum_{
x \in \Gamma_n} \frac{d(o,xo)}{n}  \eqno(1.2)
\]
converge to some $\lambda >0$, as $n \to \infty$
 \cite{pollicott1998comparison}, \cite{pollicott2001poincare},
 where the positivity follows immediately from the lower bound in (1.1).
(See Remark 1.3(i) below for a further discussion.) 
Furthermore, subject to a mild non-degeneracy condition, 
namely that the set $\{d(o,xo)- \lambda|x| \mbox{ : } x \in \Gamma\}$ is not bounded,
the distribution of $(d(o,xo)-\lambda n)/\sqrt{n}$ with respect to the normalised counting measure on $\Gamma_n$
converges to a normal distribution $N(0,\sigma^2)$, as $n \to \infty$, for some finite
$\sigma^2>0$.

In this paper, we shall consider the corresponding questions when
we restrict our group elements to a non-trivial conjugacy class.
Let $\mathfrak C$ be a non-trivial conjugacy class in $\Gamma$
and let $k = \min\{|x| \mbox{ : } x \in \mathfrak C\}$.
Let $\mathfrak C_n = \{x \in \mathfrak C \mbox{ : } |x|=n\}$
and note that $\mathfrak C_n$ is non-empty if and only if $n=k+2m$, $m \in
\mathbb Z^+$.

\begin{theorem}\label{main-equi}
We have
\[
\lim_{m \to \infty} \frac{1}{\#\mathfrak C_{k+2m}} 
\sum_{x \in \mathfrak C_{k+2m}} \frac{d(o,xo)}{k+2m}
= \lambda.
\]
\end{theorem}

Subject to an additional condition, we also have a central limit theorem.

\begin{theorem}\label{main-clt}
Suppose that 
the set $\{d(o,xo)- \lambda |x| \mbox{ : } x \in \Gamma\}$ is not bounded.
Then the
 distribution of $(d(o,xo)-\lambda (k+2m))/\sqrt{k+2m}$ with respect to normalised counting measure on 
 $\mathfrak C_{k+2m}$
converges to a normal distribution $N(0,2\sigma^2)$, as $n \to \infty$, i.e.
\begin{multline*}
\lim_{m \to \infty} \frac{1}{\#\mathfrak C_{k+2m}} \#\mleft\{x \in \mathfrak C_{k+2m}
\mbox{ : } \frac{d(o,xo)-\lambda (k+2m)}{\sqrt{k+2m}} <y\mright\}
\\
= \frac{1}{2\sqrt{\pi} \sigma}\int_{-\infty}^y e^{-t^2/4\sigma^2} \, dt.
\end{multline*}
\end{theorem}

A noteworthy feature of this result is that the variance is twice the variance that appears in
the unrestricted case. Theorems \ref{main-equi} and \ref{main-clt} will follow from more general 
results proved below.

\begin{remarks}
(i) The existence of the limit 
\[
\lim_{n \to \infty} \frac{1}{\#\Gamma_n} \sum_{x \in \Gamma_n} 
\frac{d(o,xo)}{n}
\]
follows from Proposition 8 of \cite{pollicott1998comparison}. 
The results in \cite{pollicott1998comparison} are proved for co-compact groups of isometries of real hyperbolic space but go over to co-compact groups of isometries of 
CAT($-1$) spaces by the arguments of \cite{pollicott2001poincare}.
Some explanation may be in order here.
The paper \cite{pollicott2001poincare} is 
written in the context of compact manifolds (possibly with boundary) with variable negative curvature.
In our situation, $X$ corresponds to the universal cover of $M$ and $\Gamma$ to the fundamental group,
acting as isometries on $X$. Given a point $p \in M$ and a non-identity element
$x \in \Gamma$ (thought of as $\pi_1(M,p)$), the number $l(x)$ is defined to be the length of the shortest geodesic arc from $p$ to itself in the homotopy class determined by $x$. This can be reinterpreted as the number $d(o,xo)$, where $o$ is a lift of $p$ to $X$, returning us to our original setting. 
Although the results of \cite{pollicott2001poincare} are stated for manifolds of negative curvature, the arguments used there, in particular the key Lemma 1, only require that $X$ be a CAT($-1$) space.
A consequence of this lemma is that $d(o,xo)$ can be written as the Birkhoff sum of a H\"older continuous function on an associated subshift of finite type (Proposition 3 of \cite{pollicott2001poincare}); this shows that 
$d(o,xo)$ satisfies the assumption (A1) in the next section. (Of course, the assumption (A2) below is trivially satisfied.)

The existence of a limit in (1.2) continues to hold if $\Gamma$ is a word hyperbolic group following an
observation of Calegari and Fujiwara \cite{CF}, using a result of Coornaert
\cite{coornaert1993patterson}. 

\smallskip
\noindent
(ii)
The number $\lambda>0$ may also be characterised in the following way.
Let $\Sigma$ be the space of infinite reduced words on 
$\mathcal A \cup \mathcal A^{-1}$ and let $\mu_0$ be the measure of maximal entropy 
for the shift map $\sigma : \Sigma \to \Sigma$ -- these objects are defined in section 2.
Then, for $\mu_0$-a.e. $(x_i)_{i=0}^\infty \in \Sigma$,
\[
\lim_{n \to \infty} \frac{d(o,x_0x_1 \cdots x_{n-1}o)}{n} = \lambda.
\]
This follows from the representation of $d(o,xo)$ as a Birkhoff sum of a H\"older continuous function on $\Sigma \cup \Gamma$ and the ergodic theorem.
(See, for example, Lemma 4.4 and Corollary 4.5 of \cite{pollicott2011matrix}.)

\smallskip
\noindent
(iii) The fact that the variance in Theorem \ref{main-clt} is independent of the choice of conjugacy class is
a consequence of the hypothesis that $\{d(o,xo)-\lambda |x| \hbox{ : } x \in \Gamma\}$ is unbounded,
which is a condition on the behaviour of the displacement function $d(o,xo)$ over the whole group 
$\Gamma$. (The same may be said of the assumption (A3) in the next section.)

\smallskip
\noindent
(iv)
It is interesting to have examples where the above hypothesis that 
$S=\{d(o,xo)-\lambda |x| \hbox{ : } x \in \Gamma\}$ is unbounded holds.
The hypothesis may be reformulated as follows. 
For $x \in \Gamma$, define homogeneous length
functions associated to $d(o,xo)$ and $|x|$:
\[
\ell(x) := \lim_{n \to \infty} \frac{d(o,x^no)}{n}
\quad
\mbox{and}
\quad
\|x\| := \lim_{n \to \infty} \frac{|x^n|}{n}.
\]
Then $\ell(x)$ and $\|x\|$ are positive and depend only on the conjugacy class of $x$, so we may write
$\ell(\mathfrak C)$ and $\|\mathfrak C\|$. Furthermore, $\ell(\mathfrak C)$ is the length of the closed geodesic
on the quotient $\Gamma \backslash X$ in the free homotopy class determined by $\mathfrak C$.
If $S$ were bounded then we would have 
$\ell(\mathfrak C) = \lambda \|\mathfrak C\|$ for all non-trivial conjugacy classes $\mathfrak C$. In particular,
the length spectrum of $\Gamma \backslash X$, i.e.
the set of lengths of closed geodesics, would be contained in the set $\lambda \mathbb Z$.
However, it is known that the length spectrum is not contained in a discrete subgroup of the reals when 
$X$ is the real hyperbolic space $\mathbb H^k$, $k \geq 2$ or when $X$ is a simply connected surface
of pinched variable negative curvature \cite{dalbo}, so the hypothesis holds in these cases.
More generally, though the hypothesis may fail in particular cases, it will typically hold. For example, 
if $X$ is a metric tree with quotient metric graph $\Gamma \backslash X$ then to ensure the 
hypothesis is satisfied, one only requires that $\Gamma \backslash X$ has two closed paths whose
 lengths have irrational 
ratio.

\smallskip
\noindent
(v)
The above results still hold if $d(o,xo)$ is replaced by a H\"older length function
$L(x)$ as defined in \cite{horsham2009lengths}.
\end{remarks}

We end the introduction by outlining the contents of the paper.
In section 2 we discuss the relationship between free groups and subshifts of finite type and 
state more general versions of Theorems \ref{main-equi} and \ref{main-clt}. 
In section 3 we introduce the transfer operators that we use for our analysis and discuss some of their properties. In section 4 we introduce a generating function $\eta_{\mathfrak C}(z,s)$
related to the conjugacy class $\mathfrak C$, where $z$ and $s$ are complex variables.
In the geometric setting considered above, this generating function takes the form
\[
\eta_{\mathfrak C}(z,s) = \sum_{m=0}^\infty z^{k+2m} \sum_{x \in \mathfrak C_{k+2m}}
e^{sd(o,xo)}.
\]
In particular, the variable $z$ is associated to the word length and the variable $s$ to the geometric 
length (or to a more general weighting below).
This generating function is perhaps the main new innovation of the paper, though its analysis
is inspired by work on a somewhat similar function in \cite{ks}.
This allows us to prove our first main result. We
conclude the paper in section 5 by proving a central limit theorem over a non-trivial conjugacy class.
The results in this paper form part of the first author's PhD thesis at the University of Warwick.

\section{Free groups and subshifts}

As above, let \(\Gamma\) be a free group with free generating set 
\(\mathcal{A}=\{a_1,\ldots, a_p\}\), \(p \ge 2\).  
Write $\mathcal A^{-1} = \{a_1^{-1}, \ldots, a_p^{-1}\}\).
A word \(x_0\cdots x_{n-1}\), with letters \(x_k \in \mathcal{A}\cup \mathcal{A}^{-1}\),  
is said to be \textit{reduced} if \(x_{k+1} \neq x_k^{-1}\) for each \(k\in\{0,\ldots, n-2\}\) and \textit{cyclically reduced} if, in addition,
\(x_0 \neq x_{n-1}^{-1}\).  
Every non-identity element $x \in \Gamma$ has a unique representation as
a reduced word $x = x_0 x_1 \cdots x_{n-1}$ and 
we define the \textit{word length} $|x|$ of \(x\), by $|x|=n$.
We associate to the identity element the empty word and set \(|1|=0\).  
Let \(\Gamma_n = \{x\in \Gamma\colon |x|=n\}\).  

Let \(\mathfrak{C}\) be a non-trivial conjugacy class in \(\Gamma\) and let \(k = \inf\{|x|
\colon x\in\mathfrak{C}\} >0\).  The set of elements with shortest word length in the 
conjugacy class is precisely the set of elements with cyclically reduced word 
representations.  In fact, if \(g=g_1 \cdots g_k\in\mathfrak{C}\) is cyclically reduced then 
all cyclically reduced words in \(\mathfrak{C}\) are given by cyclic permutations of the 
letters in \(g_1 \cdots g_k\).  Let \(\mathfrak{C}_n = \{ x \in \mathfrak{C} \colon |x| = n\}\) 
and note that \(\mathfrak{C}_n\) is non-empty if and only if \(n = k+2m\).  If 
\(x \in \mathfrak{C}_{k+2m}\) then its reduced word representation is of the form 
\(w_m^{-1} \cdots w_1^{-1} g_1 \cdots g_k w_1 \cdots w_m\), for some cyclically reduced 
\(g = g_1 \cdots g_k \in\mathfrak{C}_k\) and \(w= w_1 \cdots w_m \in \Gamma_m\) with 
\(w_1 \neq g_1, g_k^{-1}\).   
Hence it is convenient to
introduce the notation \(\Gamma_m(g) = \{w\in \Gamma_m \colon w_1 \neq g_1, 
g_k^{-1}\}\). A simple calculation shows that the number of elements in 
\(\mathfrak{C}_{k+2m}\) is given by 
\(\# \mathfrak{C}_{k+2m} = (2p-2)(2p-1)^{m-1} \#\mathfrak{C}_k\).

We associate to the free group \(\Gamma\) a dynamical system called
a {\it subshift of finite type}.  This subshift of finite type is formed from the space of infinite reduced words 
(with the obvious definition) adjoined to the elements of $\Gamma$ together
with the dynamics given by the action of the shift
map.
It will be convenient to describe this space by means of a transition matrix.
Define a $p \times p$ matrix $A$, with rows and columns 
indexed by \(\mathcal{A}\cup\mathcal{A}^{-1}\), by
\(A(a,b) = 0\) if \(b=a^{-1}\) {and} \(A(a,b) =1\) {otherwise.}  We then define
	\begin{equation*}
		\Sigma = \{ (x_n)_{n=0}^\infty \in (\mathcal{A}\cup\mathcal{A}^{-1})^{\mathbb{Z}^{+}} \colon A(x_n, x_{n+1})=1,\, \forall n \in \Z^+ \}.
	\end{equation*}
The shift map $\sigma : \Sigma \to \Sigma$ is defined by \((\sigma (x_n)_{n=0}^\infty) = (x_{n+1})_{n=0}^\infty\).
We give $\mathcal A \cup \mathcal A^{-1}$ the discrete topology, 
$(\mathcal{A}\cup\mathcal{A}^{-1})^{\mathbb{Z}^{+}}$ the product topology and
$\Sigma$ the subspace topology; then $\sigma$ is continuous.
Since the matrix \(A\) is \textit{aperiodic} (i.e. there exists \(n\ge 1\) such that for each pair of indices \((s,t)\), \(A^n(s,t)>0\)), 
 \(\sigma:\Sigma\to\Sigma\) is \textit{mixing} (i.e. for every pair of non-empty open sets \(U,V\subset \Sigma\) there is an \(n\in\Z^+\) such that \(\sigma^{-k} U \cap V \neq \emptyset\) for \(k\ge n\)).

We augment $\Sigma$ by defining $\Sigma^* = \Sigma \cup \Gamma$, where the elements of $\Gamma$ are identified with finite reduced words in the obvious way.
The shift map naturally extends to a map  \(\sigma: \Sigma^*\to\Sigma^*\), 
where, for the finite reduced word \( x_0 x_1\cdots x_{n-1} \in \Gamma\), we set  \(\sigma(x_0 x_1\cdots x_{n-1}) = x_1\cdots x_{n-1}\); and for the empty word \(\sigma 1 =1\). 
It is sometimes useful to think of an element of $\Gamma$ as an infinite sequence 
ending in an infinite string of $1$s.

We endow \(\Sigma^*\) with the following metric, consistent with the topololgy on 
$\Sigma$. Fix \(0<\theta<1\) then let \(d_\theta(x,x)=0\) and, for  \(x\neq y\), let \(d_\theta(x,y) = \theta^{k}\), where \(k = \min\{n\in\Z^+ \colon x_n \neq y_n\}\).  For a finite word \(x=x_0 x_1\cdots x_{m-1}\in \Gamma_m\) we take \(x_n=1\) (the empty symbol) for each \(n\ge m\).  Then \(\sigma:\Sigma^*\to\Sigma^*\) is continuous and
$\Gamma$ is a dense subset of $\Sigma^*$.

We will write $\mathcal M$ for the set of $\sigma$-invariant Borel probability measures on $\Sigma$.
For $\nu \in \mathcal M$, we write $h(\nu)$ for its entropy. We define the pressure of a continuous function
$f : \Sigma \to \mathbb R$ by
\[
P(f) := \sup_{\nu \in \mathcal M} \mleft(h(\nu) + \int f \, d\nu\mright).
\]
If $f$ is H\"older continuous then the supremum is attained at a unique $\mu_f \in \mathcal M$, called
the equilibrium state of $f$. 
(If $f : \Sigma^* \to \mathbb R$ then we write $P(f) := P(f|_\Sigma)$.)
The equilibrium state of zero $\mu_0$ is also called the measure of maximal 
entropy and $P(0)$ is equal to the topological entropy $h$ of $\sigma : \Sigma \to \Sigma$. 
It is easy to calculate that $h = \log (2p-1)$ (the logarithm of the largest eigenvalue of $A$) and that 
$\mu_0$ is characterised by 
\[
\mu_0([w]) = (2p)^{-1}(2p-1)^{-(n-1)},
\]
where, 
for a reduced word $w=w_0 w_1 \cdots w_{n-1} \in \Gamma_n$, $[w]$ is
the associated cylinder set
$[w] \subset \Sigma^*$ by
$[w] =\{(x_j)_{j=0}^\infty \mbox{ : } x_j=w_j, \, j=0,\ldots,n-1\}$.
(Technically, this defines $\mu_0$ as a measure on $\Sigma^*$ with support equal to $\Sigma$.)

Two H\"{o}lder continuous functions \(f,g:\Sigma^*\to\R\) are \textit{cohomologous} if there exists a continuous function \(u:\Sigma^*\to\R\) such that \(f=g + u\circ \sigma - u\).  
Two H\"older continuous functions have the same equilibrium state if and only if  they differ by the sum of a coboundary and a constant.
A function \(f: \Sigma^*\to\R\) is \textit{locally constant} if there exists \(n\ge1\) such that for all pairs \(x,y\in\Sigma\) with \(x_k = y_k\) for \(0\le k \le n\), \(f(x)=f(y)\).  Locally constant functions are automatically H\"{o}lder continuous for any choice of H\"older exponent.
For a function \(f:\Sigma^*\to\R\) we denote by \(f^n(x)\) the Birkhoff sum 
\[
f^n(x) := f(x) + f(\sigma x) + \cdots + f(\sigma^{n-1} x).
\]

We have the following result \cite{PP}, \cite{ruelle1978thermodynamic}.

\begin{prop} \label{derivatives}
If $f : \Sigma \to \mathbb R$ is H\"older continuous then, for $t \in \mathbb R$, $t \mapsto P(tf)$ 
is real analytic,
\[
\frac{dP(tf)}{dt}\Big|_{t=0} = \int f \, d\mu_0
\]
and 
\[
\frac{d^2P(tf)}{dt^2}\Big|_{t=0} =   \sigma_f^2 :=
\lim_{n \to \infty} \frac{1}{n} \int \mleft(f^n(x) - n\int f \, d\mu_0\mright)^2 \, d\mu_0.
\]
Furthermore, $\sigma_f^2=0$ if and only if $f$ is cohomologous to a constant.
\end{prop}

For convenience, in the work that follows we shall interchangeably refer to elements \(x\in \Gamma\) and the associated element of the sequence space \(x\in\Sigma^*\).  We now state the technical result from which Theorem \ref{main-equi} follows.
We consider functions \(F :\Gamma\to\R\)
which satisfy the following two assumptions.
	\begin{enumerate}[label=(A\arabic*), noitemsep]
		\item \label{A1} There exists a H\"{o}lder continuous function \(f:\Sigma^*\to\R\) so that \(F(x) = f^n(x)\) for each \(x\in \Gamma_n\) with \(n\ge 0\), and
		\item \label{A2} \(F(x) = F(x^{-1})\).
	\end{enumerate}

We will prove the following.

\begin{theorem} \label{equi-holder}  Suppose that \(F :\Gamma\to\R\) satisfies assumptions \ref{A1} and \ref{A2}. There exists $\overline{F} \in \mathbb R$ such that
	\begin{equation*}
		\lim_{m\to\infty} \frac{1}{\#\mathfrak{C}_{k+2m}} \sum_{x\in\mathfrak{C}_{k+2m}} \frac{F(x)}{k+2m} = \overline{F}.
	\end{equation*}
Furthermore, $\overline{F} = \int f \, d\mu_0$.
\end{theorem}	

We remark that, without the restriction to a conjugacy class, the analogous result 
\[
\lim_{n \to \infty} \frac{1}{\#\Gamma_n}
\sum_{x \in \Gamma_n} \frac{F(x)}{n} = \overline{F}
\]
holds subject only to \ref{A1}. This follows from the analysis in \cite{pollicott1998comparison} or from a large deviations argument following the ideas of Kifer \cite{kifer} as employed in \cite{pollicott1996large}.

We also establish a central limit theorem for the group elements in \(\Gamma\) restricted to a non-trivial conjugacy class.  In addition to assumptions \ref{A1} and \ref{A2}, we require a third assumption.  
	\begin{enumerate}[label=(A\arabic*)] \setcounter{enumi}{2}
		\item \label{A3} $F(\cdot) - \overline{F}|\cdot|$ is unbounded as a function from 
		$\Gamma$ to $\mathbb R$.
	\end{enumerate}

\begin{lem}
Let $F$ and $f$ be as in \ref{A1}. Then
$F(\cdot) -\overline{F}|\cdot|$ is bounded if and only if $f|_\Sigma$ is cohomologous to a constant.
\end{lem} 	

\begin{proof}
For simplicity, we will write $f|_\Sigma=f$.
If $F(\cdot) -\overline{F}|\cdot|$ is bounded then $\mleft\{f^n(x) - n\int f \, d\mu_0 \mbox{ : } 
x \in \Gamma_n, \ n \geq 1\mright\}$ is a bounded
set. Since $f$ is H\"older continuous, this implies that 
\[
\mleft\{f^n(x) - n\int f \, d\mu_0 \mbox{ : } x \in \Sigma,
\ n \geq 1\mright\}
\]  
is also bounded. In particular, $\left(f^n - n\int f \, d\mu_0\right)^2/n$ converges 
uniformly to zero and it is
easy to deduce that $\sigma_f^2 =0$. Therefore, by Proposition \ref{derivatives}, $f$ is cohomologous
to a constant.

On the other hand, if $f$ is cohomologous to a constant then, again by H\"older continuity,
$\{F(x)-\overline{F}|x| \mbox{ : } x \in \Gamma\} = \left\{f^n(x) - n\int f \, d\mu_0 \mbox{ : } 
x \in \Gamma_n, \ n \geq 1\right\}$ is bounded.
\end{proof}
	
It is a well-known result that if $f : \Sigma \to \mathbb R$ is not cohomologous to a constant then
the process $f \circ \sigma^n$, $n\geq 1$, satisfies 
a central limit theorem with respect to $\mu_0$ with variance $\sigma_f^2$, i.e., 
that $\left(f^n - n\int f \, d\mu_0\right)/\sqrt{n}$ converges in distribution to a normal random variable with 
mean zero and variance $\sigma_f^2>0$ or, explicitly, that
for $a \in \mathbb R$,
\[
\lim_{n \to \infty} \mu_0\mleft\{x \in \Sigma : \mleft(f^n(x) -n\int f \, d\mu_0\mright)/\sqrt{n} \leq a\mright\}
= \frac{1}{\sqrt{2\pi} \sigma_f} \int_{-\infty}^a e^{-u^2/2\sigma_f^2} \, du
\]
\cite{ruelle1978thermodynamic}, \cite{coelho1990central}. Furthermore, analogues of this hold for the periodic
points of $\sigma : \Sigma \to \Sigma$ \cite{coelho1990central} and, by adapting the proof, for pre-images
of a given point. This gives a central limit theorem for $F$ over $\Gamma_n$ (without the assumption \ref{A2}). Particular cases of this have appeared in articles
by Rivin
\cite{rivin2010growth} for homomorphisms, and Horsham and Sharp 
\cite{horsham2009lengths} (see also \cite{horsham2008central}) for quasimorphisms.
Calegari and Fujiwara \cite{CF} prove a central limit theorem for quasimorphisms on 
Gromov hyperbolic groups, but have more restrictions on the regularity
of the quasimorphism.
Restricting to a non-trivial conjugacy class, we have the following theorem.
	
\begin{theorem} \label{clt-holder} Suppose that $F : \Gamma \to \mathbb R$ satisfies
assumptions  \ref{A1}, \ref{A2} and \ref{A3}.  Then 
the sequence 
	\begin{equation*}
		 \frac{1}{\# \mathfrak{C}_{k+2m}} \#\mleft\{x\in \mathfrak{C}_{k+2m} 
		 \colon (F(x)-(k+2m)\overline{F})/\sqrt{{k+2m}} \le a \mright\}
	\end{equation*}
converges to the distribution function of a normal random variable with mean \(0\) and positive variance \(2\sigma_f^2\).
\end{theorem}

We note the limiting distribution function is independent of the choice of non-trivial conjugacy class.  Further, it is interesting that the variance in Theorem \ref{clt-holder} is twice the variance when we do not restrict elements \(x\in \Gamma\) to a non-trivial conjugacy class.

\begin{proof}[Proof of Theorems \ref{main-equi} and \ref{main-clt}]
As in the introduction, let the free group $\Gamma$ act convex co-compactly on a 
CAT$(-1)$ space $(X,d)$. Then it was shown in \cite{pollicott2001poincare}
that $F(x) := d(o,xo)$ satisfies \ref{A1}. (In fact, the result in 
\cite{pollicott2001poincare} is stated when $X$ is a simply connected manifold
with bounded negative curvatures but the proof only requires the CAT$(-1)$
property.) Assumption \ref{A2} is clearly satisfied. Therefore, Theorem \ref{main-equi} follows from Theorem \ref{equi-holder}.
Furthermore, the additional assumption on $d(o,xo)$ in Theorem \ref{main-clt}
matches \ref{A3} and so Theorem \ref{main-clt} also follows.
\end{proof}

\section{Transfer operators}

  In this section we recall results from the theory of transfer operators that will be used to deduce Theorem \ref{equi-holder} and 
Theorem \ref{clt-holder}.
Let $\mathcal{F}_\theta(\Sigma,\mathbb C)$ denote the space of $d_\theta$-Lipschitz functions
$f : \Sigma \to \mathbb C$. This is a Banach space with respect to the norm
$\|\cdot\|_\theta = \|\cdot\|_\infty + |\cdot|_\theta$, where
\[
|f|_\theta:= \sup_{x \neq y} \frac{|f(x)-f(y)|}{d_\theta(x,y)}.
\]
Any H\"older continuous function becomes Lipschitz by changing the choice of $\theta$
(i.e. if $f$ has H\"older exponent $\alpha$ with respect to $d_\theta$ then 
$f \in \mathcal{F}_{\theta^\alpha}(\Sigma,\mathbb C)$), so 
there is no loss of generality in restricting to these spaces.
Given \(g \in \mathcal{F}_\theta(\Sigma,\C)\), the \textit{transfer operator} \(L_g: \mathcal{F}_\theta(\Sigma,\C)
\to \mathcal{F}_\theta(\Sigma,\C)\) is defined pointwise by
	\begin{equation*}
			L_g \omega(x) = \sum_{\sigma y = x} e^{g(y)} \omega(y).
	\end{equation*}

We have the following standard result \cite{PP}, \cite{ruelle1978thermodynamic}.

\begin{prop}  [Ruelle--Perron--Frobenius Theorem] \label{rpf}
Suppose that \(g \in \mathcal{F}_\theta(\Sigma,\C)\) is real-valued. Then 
\(L_g: \mathcal{F}_\theta(\Sigma,\C)
\to \mathcal{F}_\theta(\Sigma,\C)\) has a simple eigenvalue equal to $e^{P(g)}$, associated strictly positive eigenfunction $\psi$ and eigenmeasure $\nu$ (i.e. $L_g\psi = e^{P(g)}\psi$ and $L_g^*\nu =e^{P(g)}\nu$),
normalised so that $\nu$ is a probability measure and $\int \psi \, d\nu =1$.
Furthermore, the rest of the spectrum of $L_g$ is contained in a disk of radius strictly smaller than $e^{P(g)}$.
\end{prop}	
	
The equilibrium state $\mu_g$ is given by $d\mu_g = \psi d\nu$.	
We say that $g$ is {\it normalised} if $L_g1=1$ (which in particular implies $P(g)=0$).	
If we replace \(g\) by \(g' = g - P(g) + u - u\circ \sigma\) where \(u = \log \psi\) then \(g'\) is normalised
and \(g\) and \(g'\) have the same equilibrium state.

Suppose that \(f,g \in \mathcal{F}_\theta(\Sigma,\mathbb C)\) are real-valued functions.  We consider small perturbations of the operator \(L_{g}\) of the form \(L_{g+sf}\) for values of \(s\in \mathbb{C}\) in a neighbourhood of the origin.  Since \(e^{P(g)}\) is a simple isolated eigenvalue of \(L_g\), for small perturbations of \(s\) close to the origin this eigenvalue persists so that the operator \(L_{g+sf}\) has a simple eigenvalue 
\(\beta(s)\) and corresponding eigenfuction $\psi_s$ that vary analytically with \(s\) 
and satisfy $\beta(0) = e^{P(g)}$ and $\psi_0=\psi$ \cite{Kato}.
Furthermore, by the upper semi-continuity of the spectral radius, there exists $\varepsilon>0$ such that,
for \(s\) close to the origin,
 the remainder of the spectrum of $L_{g+sf}$ lies in a disk of radius 
 $e^{P(g) -\varepsilon}$.
 We extend the definition of pressure by setting \(e^{P(g+sf)}= \beta(s)\).

We find it useful to consider \(\sigma: \Sigma^*\to \Sigma^*\) as a subshift of finite type and will use the previous notation and concepts introduced for \(\Sigma\) in this setting.   We modify the definition of the {transfer operator} 
\(L_{sf}: \mathcal F_\theta(\Sigma^*, \C)\to \mathcal F_\theta(\Sigma^*, \C)\) as follows:
	\begin{equation*}
		L_{sf} \omega(x) = \sum_{\substack{\sigma y = x \\ y\neq 1}} e^{sf(y)} \omega(y).
	\end{equation*}

Here \(1\) denotes the identity element in \(\Gamma\), considered as an infinite
word $(1,1,\ldots)$. We note the transfer operator we use differs from the usual 
definition by excluding the preimage \(y=1\) from the summation over the set 
\(\{y\in\Sigma^* \colon \sigma y =x\}\); however, the definition of this transfer 
operator agrees with our previous definition for each \(x\neq 1\).  
Following Lemma 2 of \cite{pollicott1998comparison}, 
$L_{sf}: \mathcal F_\theta(\Sigma^*, \C)\to \mathcal F_\theta(\Sigma^*, \C)$
has the same isolated eigenvalues as 
$L_{sf} :\mathcal F_\theta(\Sigma \cup \{1\}, \C)\to \mathcal F_\theta(\Sigma \cup \{1\}, \C)$.
Since the modified definition of $L_{sf}$ 
excludes the eigenvalue \(e^{sf(1)}\) associated to the eigenfunction 
\(\chi_{\{1\}}\) (the indicator function of the set $\{1\}$),
$L_{sf}: \mathcal F_\theta(\Sigma^*, \C)\to \mathcal F_\theta(\Sigma^*, \C)$
therefore has the same isolated eigenvalues as
$L_{sf}: \mathcal F_\theta(\Sigma, \C)\to \mathcal F_\theta(\Sigma, \C)$.
Furthermore, again by Lemma 2 of \cite{pollicott1998comparison}, 
$L_{sf} : \mathcal F_\theta(\Sigma^*, \C)\to \mathcal F_\theta(\Sigma^*, \C)$
is quasi-compact with essential spectral radius at most $\theta e^{P(\mathrm{Re}(s) f)}$,
and so it suffices to consider the spectral theory of $L_{sf}$ on $\mathcal F_\theta(\Sigma,\mathbb C)$.

\section{Proof of Theorem \ref{equi-holder}}


%
%
In this section, we will prove Theorem \ref{equi-holder}.
We introduce a generating function  \(\eta_\mathfrak{C}(z,s)\) on two complex variables given by
	\begin{equation*}
		\eta_{\mathfrak{C}}(z,s) = \sum_{m=0}^\infty z^{k+2m} \sum_{x\in \mathfrak{C}_{k+2m}} e^{sF(x)} = \sum_{m=0}^\infty z^{k+2m} \sum_{g\in\mathfrak{C}_k} \sum_{w\in \Gamma_m(g)} e^{sf^{k+2m}(w^{-1}gw)}
	\end{equation*}
(wherever the series converges).
We prove the theorem by studying the asymptotic behaviour, as \(m\to\infty\), of the coefficient of \(z^{k+2m}\) in the power series
	\begin{equation*}
		\mleft. \frac{\partial}{\partial s} \eta_\mathfrak{C}(z,s) \mright|_{s=0} = \sum_{m=0}^\infty z^{k+2m} \sum_{x\in \mathfrak{C}_{k+2m}} F(x).
	\end{equation*}

We will find the following bound useful in the proof of Theorem \ref{equi-holder}.
	\begin{lem} \label{lem: approx} Suppose that \(f\in\mathcal{F}_{\theta}(\Sigma^*,\mathbb C)\), \(g\in \mathfrak{C}_k\) and \(w\in \Gamma_m(g)\) then there exists a constant \(K>0\), independent of \(m\), such that
	\begin{equation*}
		|f^{k+2m}(w^{-1}gw) - f^m(w) - f^k(g) -f^m(w^{-1})| \le K.
	\end{equation*}
	\end{lem}
\begin{proof}
 We have \(f^{k+2m}(w^{-1}gw) = f^{m}(w^{-1}gw) + f^k(gw) + f^m(w)\). Thus
	\begin{multline*}
		|f^{k+2m}(w^{-1}gw) - f^m(w) - f^k(g) -f^m(w^{-1})| \\ \le |f^m(w^{-1}gw) - f^m(w^{-1})| + |f^k(gw) - f^k(g)| \le \frac{2|f|_\theta \theta}{1-\theta}
	\end{multline*}
and we are done.
\end{proof}

By Lemma \ref{lem: approx},
	\begin{equation*}
		\exp ( sf^{k+2m}(w^{-1}gw)) = \exp \mleft( s( f^m(w) + f^k(g) +f^m(w^{-1}) ) \mright) + s\kappa_w + \xi_w(s)
	\end{equation*}
where \(\kappa_w = f^{k+2m}(w^{-1}gw) - f^m(w) - f^k(g) -f^m(w^{-1})\) is uniformly bounded for $w \in  \Gamma$ (by Lemma \ref{lem: approx}) and \(\xi_w(s) = s^2 \zeta_w(s)\), with \(\zeta_w(s)\) an entire function.  By this approximation and assumption \ref{A2}, we have
	\begin{equation*}
		\eta_{\mathfrak{C}}(z,s) =  \sum_{m=0}^\infty z^{k+2m} \sum_{g\in \mathfrak{C}_{k}} \sum_{w\in \Gamma_m(g)} e^{s(f^k(g) + 2f^m(w))} +\delta(z,s)
\end{equation*}
where
	\begin{equation*}
		\delta(z,s) = \sum_{m=0}^\infty z^{k+2m} \sum_{g\in \mathfrak{C}_{k}} \sum_{w\in \Gamma_m(g)} \left(s\kappa_w + \xi_w(s)\right).
	\end{equation*} 
%


Let \(\chi_g : \Sigma^* \to \R\) be the locally constant function given by
	\begin{equation*}
		\chi_g((w_n)_{n=0}^\infty) = \begin{dcases} 0 & \text{if}\ w_0 = g_1, g_k^{-1},\, \text{and} \\
			1 & \text{otherwise}.
			\end{dcases}
	\end{equation*}
We introduced the function \(\chi_g\) in order to write \(\eta_\mathfrak{C}(z,s)\) in terms of the transfer operator.  We have
\begin{align*}
	\eta_{\mathfrak{C}}(z,s) &= \sum_{g\in \mathfrak{C}_{k}} e^{sf^k(g)} \sum_{m=0}^\infty z^{k+2m} \sum_{w\in \Gamma_m} e^{2sf^m(w)} \chi_g(w) + \delta(z,s), \\
		&=  \sum_{g\in \mathfrak{C}_{k}} e^{sf^k(g)} \sum_{m=0}^\infty z^{k+2m} (L_{2sf}^m \chi_g)(1) + \delta(z,s). 
\end{align*}

Thus the power series \(\sum_{m=0}^\infty z^{k+2m} \sum_{x\in \mathfrak{C}_{k+2m}} F(x)\) can be written in terms of the transfer operator since
\begin{multline*}
	\mleft. \frac{\partial}{\partial s}\, \eta_{\mathfrak{C}}(z,s)\mright|_{s=0} =
	\sum_{g\in \mathfrak{C}_{k}} \sum_{m=0}^\infty \mleft. \frac{\partial}{\partial s}\, z^{k+2m} (L_{2sf}^m \chi_g)(1)  \mright|_{s=0} \\
	+ \sum_{g\in \mathfrak{C}_{k}} f^k(g) \sum_{m=0}^\infty z^{k+2m} (L_{0}^m \chi_g)(1) 
	 + \mleft. \frac{\partial}{\partial s}\, \delta(z,s)\mright|_{s=0}.
\end{multline*}

We analyse the growth of the coefficients of the power series in the following sequence of lemmas.

\begin{lem} The coefficient of \(z^{k+2m}\) in the power series \(\sum_{m=0}^\infty z^{k+2m} (L_{0}^m \chi_g)(1)\) grow with order \(O(e^{mh})\).
\end{lem}

The coefficient in the next lemma grows with the same order. 
\begin{lem} \label{lem: 5.3.4}
	The coefficient of \(z^{k+2m}\) in the power series \(\mleft. \frac{\partial}{\partial s} \delta(z,s) \mright|_{s=0}\) grow with order \(O(e^{mh})\).
\end{lem}
\begin{proof}
Since, for each $w \in \Gamma$, $\xi_w'(0)=0$,
\begin{equation*} \mleft. \frac{\partial}{\partial s} \delta(z,s) \mright|_{s=0} = \sum_{m=0}^\infty z^{k+2m} \sum_{g\in \mathfrak{C}_k} \sum_{w\in \Gamma_m(g)} \kappa_w.
\end{equation*}
For each \(w\in \Gamma\) we have \(|\kappa_w| \le K\).  Thus the coefficient of \(z^{k+2m}\) is bounded in modulus by
\begin{equation*}
	\sum_{g\in \mathfrak{C}_k} \sum_{w\in \Gamma_m(g)} K = K \# \mathfrak{C}_{k+2m} = K(2p-2)(2p-1)^{m-1} \#\mathfrak{C}_k = O(e^{mh}),
\end{equation*}
from which the lemma follows.
\end{proof}

We decompose the transfer operator \(L_{sf}\) into the projection \(R_s\) associated to the eigenspace associated to the eigenvalue \(e^{P(sf)}\) and \(Q_s = L_{sf} - e^{P(sf)}R_s\).  For \(s\in\C\) in a neighbourhood of \(s=0\), the operators \(R_s\) and \(Q_s\) are analytic. 
We use this operator decomposition to obtain the estimates in the next two lemmas.
\begin{lem}
	The coefficient of \(z^{k+2m}\) in the power series 
	\begin{equation*} 
	\mleft. \frac{\partial}{\partial s} \sum_{g\in\mathfrak{C}_k} 
	\sum_{m=0}^\infty z^{k+2m} Q_{2s}^m \chi_g(1) \mright|_{s=0}
	\end{equation*} 
grow with order 
\(O(e^{m(h - \varepsilon)})\), for some \(\varepsilon>0\).
\end{lem}
\begin{proof}
Suppose that \(s\in \mathbb{C}\) such that \(0\le |s|< \delta_1\) then, 
as discussed in section 3, if \(\delta_1\) is sufficiently small each perturbed 
operator \(L_{2sf}\) has a simple maximal eigenvalue \(e^{P(2sf)}\).
Moreover, for \(|s|<\delta_1\), there exists \(\varepsilon_1(\delta_1)>0\) such that
\begin{equation*}
	\limsup_{m\to\infty} \| Q_{2s}^m \|^{1/m} \le e^{h-\varepsilon_1}.
\end{equation*}

We consider the analyticity of the series
	\begin{equation*}
		\sum_{g\in\mathfrak{C}_k} \sum_{m=0}^\infty z^{k+2m} Q_{2s}^m \chi_g(1).
	\end{equation*}
Suppose that we fix \(z\in\mathbb{C}\) such that \(|z|<e^{-h+\varepsilon_1}\), then 
the series converges for each \(s\in \mathbb{C}\) with \(|s|<\delta_1\).  Meanwhile, 
given \(s\in\mathbb{C}\) such that \(|s|<\delta_1\), the series converges for each 
\(z\in\mathbb{C}\) with \(|z| < e^{-h+\varepsilon_1}\).  Thus, by Hartogs' 
theorem (Theorem 1.2.5, \cite{krantz2001function}), the series converges to an 
analytic function in the 
polydisk 
\(\{s\in \mathbb{C} \colon |s|<\delta_1\}\times \{z\in\mathbb{C}\colon |z| < e^{-h+\varepsilon_1}\}\).  Thus the power series
\begin{equation*} \mleft. \frac{\partial}{\partial s} \sum_{g\in\mathfrak{C}_k} \sum_{m=0}^\infty z^{k+2m} Q_{2s}^m \chi_g(1) \mright|_{s=0}\end{equation*}
is analytic for \(|z|<e^{-h+\varepsilon_1}\) and so we estimate the coefficients of the power series by \(O(e^{m(h - \varepsilon)})\) with \(0<\varepsilon<\varepsilon_1\).
\end{proof}

There is one power series left to study.
\begin{lem}
	Let \(P'(0)\) denote the derivative of the function \(P(sf)\) evaluated at \(s=0\). The coefficient of \(z^{k+2m}\) in the power series 
	\begin{equation*} \mleft. \frac{\partial}{\partial s} \sum_{m=0}^\infty z^{k+2m} e^{mP(2sf)} R_{2s} \chi_g(1) \mright|_{s=0} \end{equation*}
is \(2me^{mh} P'(0) R_{0} \chi_g(1) + 
e^{mh} \mleft. \frac{d}{ds} R_{2s} \chi_g(1)\mright|_{s=0}\).
\end{lem}
\begin{proof} We have
	\begin{multline*} \mleft. \frac{\partial}{\partial s} \sum_{m=0}^\infty z^{k+2m} e^{mP(2sf)} R_{2s} \chi_g(1) \mright|_{s=0} \\  
		= \sum_{m=0}^\infty z^{k+2m} 2me^{mh} P'(0) R_{0} \chi_g(1) 
		+ \sum_{m=0}^\infty z^{k+2m} e^{mh} \mleft. \frac{d}{ds} R_{2s} \chi_g(1) \mright|_{s=0},
	\end{multline*}
from which the result follows. \end{proof}

Combining the above lemmas, we find that the coefficient of $z^{k+2m}$ in \(\mleft. \tfrac{\partial}{\partial s} \eta_\mathfrak{C}(z,s) \mright|_{s=0}\) satisfies the estimate
	\begin{equation*}
		\sum_{g\in\mathfrak{C}_k} 2me^{mh} P'(0) R_{0} \chi_g(1) + O(e^{mh}).
	\end{equation*}
Returning to Theorem \ref{equi-holder} we now have
	\begin{equation*}
		\frac{1}{\#\mathfrak{C}_{k+2m}} \sum_{x\in \mathfrak{C}_{k+2m}} \frac{F(x)}{k+2m} = \frac{2m}{k+2m} P'(0) \frac{e^{mh}}{\#\mathfrak{C}_{k+2m}} \sum_{g\in \mathfrak{C}_k} R_0\chi_g(1) + O\mleft(\frac{1}{m}\mright).
	\end{equation*}
Thus we have
	\begin{equation*}
		\lim_{m\to\infty} e^{-mh} \sum_{x\in\mathfrak{C}_{k+2m}} \frac{F(x)}{k+2m} = \int f\, d\mu_0 \sum_{g\in\mathfrak{C}_k} R_0\chi_g(1).
	\end{equation*}
If we substitute \(f:\Sigma^*\to\R\) given by \(f(x)=1\) for each \(x\in\Sigma^*\) into the preceding limit we obtain
	\begin{equation*}
		\lim_{m\to\infty} \frac{\#\mathfrak{C}_{k+2m}}{ e^{mh}} = \sum_{g\in\mathfrak{C}_k} R_0 \chi_g(1).
	\end{equation*}
Hence we have the desired result,
	\begin{equation*}
		\lim_{m\to\infty} \frac{1}{\#\mathfrak{C}_{k+2m}} \sum_{x\in\mathfrak{C}_{k+2m}} \frac{F(x)}{k+2m} = \int f\, d\mu_0.
	\end{equation*}

\section{Proof of Theorem \ref{clt-holder}} \label{sec: CLT}

In this section we will prove Theorem \ref{clt-holder}. By Levy's Continuity Theorem
(cf.\ Theorem 2, Chapter XV \S3, 
\cite{feller1971introduction}),
the theorem
will follow if we show that the characteristic functions
\begin{equation*}
		\varphi_m(t) = \frac{1}{\#\mathfrak{C}_{k+2m}} 
		\sum_{x\in \mathfrak{C}_{k+2m}} e^{i t (F(x)-(k+2m)\overline{F})/\sqrt{k+2m}}. 
	\end{equation*}
converge pointwise to $e^{-\sigma_f^2 t^2}$, the characteristic function of the normal 
distribution with mean zero and variance $2\sigma_f^2$.

Suppose that $F$ satisfies \ref{A1}, \ref{A2} and \ref{A3}.
By replacing $F$ with $F - \overline{F}|\cdot|$ (which still satisfies 
the three assumptions) or, equivalently, $f$ with $f - \int f \, d\mu_0$, we 
may assume without loss of generality that $\int f \, d\mu_0=0$.
This reduction does not change the variance.
We may then write
	\begin{equation*}
		\varphi_m(t) = \frac{1}{\#\mathfrak{C}_{k+2m}} \sum_{x\in \mathfrak{C}_{k+2m}} e^{i t f^{k+2m}(x)/\sqrt{k+2m}}.
	\end{equation*}
We recall the approximation, which we obtain from Lemma \ref{lem: approx},
\begin{equation*}
	\exp ( sf^{k+2m}(w^{-1}gw)) = \exp \mleft( s( 2f^m(w) + f^k(g)) \mright) 
	+ s\kappa_w + \xi_w(s),
\end{equation*}
where \(\kappa_w = f^{k+2m}(w^{-1}gw) - 2f^m(w) - f^k(g)\) is uniformly bounded for 
\(w\in \Gamma\) and \(\xi_w(s)\) is an entire function such that \(\xi_w(0)=0\).  
Using the above approximation, we write \(\varphi_m(t)\) as the sum of a leading term 
and an error term:
		\begin{equation*}
			\frac{1}{\# \mathfrak{C}_{k+2m}} \sum_{g\in\mathfrak{C}_{k}} e^{\tau f^k(g) /2} \sum_{w\in \Gamma_m(g)} e^{\tau f^m(w)} + \rho_m(t),
		\end{equation*}
where \(\tau = 2i t/\sqrt{k+2m}\) and the error term \(\rho_m(t)\) is given by
	\begin{equation*}
		\rho_m(t) = \frac{1}{\#\mathfrak{C}_{k+2m}} \sum_{g\in \mathfrak{C}_{k}} \sum_{w\in \Gamma_m(g)} \frac{i t \kappa_w}{\sqrt{k+2m}} 
		+ \xi_w \left(\frac{i t}{\sqrt{k+2m}}\right).
	\end{equation*}
Since the bound on \(\kappa_w\) is uniform and \(\xi_w(0)=0\), we find that \(\rho_m(t) \to 0\) as \(m\to\infty\).  We rewrite the leading term using the transfer operator
as
	\begin{equation*}
		\frac{1}{\# \mathfrak{C}_{k+2m}} \sum_{g\in \mathfrak{C}_k} 
		e^{\tau f^k(g)/2} L_{\tau f}^m \chi_g(1).
	\end{equation*}
For sufficiently large \(m\), the simple maximal eigenvalue \(e^{P(\tau f)}\) of the perturbed operator \(L_{\tau f}\) persists and also plays a crucial role in determining the limit of \(\varphi_m(t)\) as \(m\to\infty\).  Before we establish the limit, we first analyse the pressure function and establish a preliminary limit for \(e^{m(P(\tau f)-h)}\) as \(m\to\infty\).

Recall that the pressure function \(P(sf)\) 
(defined as the principal branch of the logarithm of $e^{P(sf)}$) is analytic in a neighbourhood 
of \(s=0\)
and that $P'(0)=\int f \, d\mu_0=0$.  
By analyticity we can choose \(\delta>0\) such that if \(|s|<\delta\) then
	\begin{equation*}
		P(2sf) 
		= h+ 2 \sigma_f^2 s^2 + s^3\vartheta(s),
	\end{equation*}
for some function \(\vartheta(s)\) that is analytic in a neighbourhood of \(s=0\). 
For sufficiently large \(m\), with \(\tau = 2i t/\sqrt{k+2m}\) as before, we have
	\begin{equation*}
		(k+2m)P\mleft( \tau f \mright) = (k+2m)h - 2\sigma_f^2 t^2 - \frac{4i t^3 \vartheta(\tau)}{3\sqrt{k+2m}}
	\end{equation*}
and so
	\begin{equation*}
		\frac{e^{(k+2m)P(\tau f)}}{e^{(k+2m)h}} = e^{-2\sigma_f^2 t^2} \exp \mleft\{ - \frac{4i t^3 \vartheta(\tau)}{3\sqrt{k+2m}} 	\mright\},
	\end{equation*}
from which the next proposition and corollary follow.

\begin{prop} \label{prop: limitevalue} We have the following limit
	\begin{equation*}
		\lim_{m\to\infty} \frac{e^{(k+2m)P(\tau f)}}{e^{(k+2m)h}} = e^{-2\sigma_f^2 t^2}.
	\end{equation*}
\end{prop}

\begin{cor} \label{cor: limitevalue} We have the limit
	\begin{equation*}
		\lim_{m\to\infty} \frac{e^{mP(\tau)}}{e^{mh}} = e^{-\sigma_f^2 t^2}.
	\end{equation*}
\end{cor}

We use the notation \(\beta(\tau) = e^{P(\tau f)}\) and \(\beta(0)=e^h\) in the proof of 
Proposition \ref{prop: fourier}.

\begin{prop} \label{prop: fourier}  The limit of \(\varphi_m(t)\) as \(m\to\infty\) is \(e^{-\sigma_f^2 t^2}\). 
\end{prop}
\begin{proof}
Written in terms of the transfer operator and a null sequence \((\rho_m(t))_{m=0}^\infty\), the function \(\varphi_m(t)\) is equal to 
	\begin{equation*}
		\frac{1}{\# \mathfrak{C}_{k+2m}} \sum_{g\in \mathfrak{C}_k} 
		e^{\tau f^k(g)/2} L_{\tau f}^m \chi_g(1) +\rho_m(t).
	\end{equation*}
We recall the decomposition of the transfer operator into \(L_{sf} = \beta(s) R_s + Q_s\).
For sufficiently large \(m\),  the leading term is given by 
	\begin{equation*}
		\frac{\beta(\tau)^m}{\# \mathfrak{C}_{k+2m}} \sum_{g\in \mathfrak{C}_k} e^{\tau f^k(g)/2} R_\tau \chi_g(1)  + \frac{1}{\# \mathfrak{C}_{k+2m}} \sum_{g\in \mathfrak{C}_k} e^{\tau f^k(g)/2} Q_\tau^m \chi_g(1).
	\end{equation*} 

 Since the spectral radius of \(Q_\tau\) is strictly less than \(|\beta(\tau)|\), we find \(\| \beta(\tau)^{-m} Q_\tau^m\| = O(\kappa^m)\) for some \(\kappa\in (0,1)\) and so we have
	\begin{equation*}
		\frac{1}{\#\mathfrak{C}_{k+2m}} \sum_{g\in\mathfrak{C}_k} e^{\tau f^k(g)/2} Q_\tau^m \chi_g(1) = O\mleft( \frac{\beta(\tau)^m}{\beta(0)^m} \kappa^m	\mright).
	\end{equation*}
By Corollary \ref{cor: limitevalue} we have \(\lim_{m\to\infty} \beta(\tau)^m/\beta(0)^m = e^{-\sigma_f^2 t^2}\) and so
	\begin{equation*}
		\lim_{m\to\infty} \frac{1}{\#\mathfrak{C}_{k+2m}} \sum_{g\in\mathfrak{C}_k} e^{\tau f^k(g)/2} Q_\tau^m \chi_g(1) =0.
	\end{equation*}

We now turn our attention to the asymptotics for the term
	\begin{equation*}
		\frac{\beta(\tau)^m}{\#\mathfrak{C}_{k+2m}} \sum_{g\in \mathfrak{C}_{k}} e^{\tau f^k(g)/2} R_{\tau} \chi_g(1).
	\end{equation*}
In order to approximate this term, we first write the projection \(R_\tau\) in terms of \(R_0\).  Since the projection is analytic for $\tau$ in a neighbourhood of \(0\) we have, 
for sufficiently large \(m\), 
\(e^{\tau f^k(g)/2} R_\tau \chi_g(1) = R_0 \chi_g (1) + O(t/\sqrt{k+2m})\).  
We recall that \(\#\mathfrak{C}_{k+2m} = (\beta(0)-1)\beta(0)^{m-1}\#\mathfrak{C}_k\) and so
	\begin{equation*}
			\frac{\beta(\tau)^m}{\#\mathfrak{C}_{k+2m}} 
			\sum_{g\in \mathfrak{C}_{k}} e^{\tau f^k(g)/2} R_{\tau} \chi_g(1) = \frac{\beta(\tau)^m}{\#\mathfrak{C}_{k+2m}} \sum_{g\in\mathfrak{C}_k} R_0\chi_g(1) + O\mleft( \frac{\beta(\tau)^m t}{\beta(0)^{m}\sqrt{k+2m}}\mright). 
	\end{equation*}

We recall the limit
	\begin{equation*}
		\lim_{m\to\infty} \frac{\#\mathfrak{C}_{k+2m}}{\beta(0)^m} = \sum_{g\in\mathfrak{C}_k} R_0 \chi_g (1)
	\end{equation*}
and so, together with the above approximation, we find the limit of \(\varphi_m(t)\) as \(m\to\infty\) is given by
	\begin{equation*}
		 \lim_{m\to\infty} \frac{\beta(\tau)^m}{\#\mathfrak{C}_{k+2m}} \sum_{g\in \mathfrak{C}_{k}} e^{\tau f^k(g)/2} R_{\tau} \chi_g(1) = \lim_{m\to \infty} \frac{\beta(\tau)^m}{\beta(0)^m} = e^{-\sigma_f^2 t^2},
	\end{equation*}
which is the desired result.
\end{proof}

\end{document}